\documentclass[12pt,reqno,a4paper,twoside]{amsart}
\usepackage{amssymb,amsmath,amscd,amstext,amsthm,amsfonts}
\usepackage[usenames]{color}
\usepackage{graphicx}
\usepackage{amsmath}
\usepackage{amssymb}
\usepackage{amsfonts}
\usepackage{multicol}
\usepackage[ansinew]{inputenc}
\usepackage{amscd}
\usepackage[mathscr]{eucal}

\newcommand{\R}{\mathbb{R}}
\newcommand{\Pp}{\mathbb{P}}

\newcommand{\E}{\mathbb{E}}

\newcommand{\Dsk}{\mathbb{D}}

\newcommand{\N}{\mathbb{N}}
\newcommand{\Z}{\mathbb{Z}}

\newcommand{\Prob}{\mbox{Prob}}

\newcommand{\Bscr}{\mathscr{B}}
\newcommand{\Cscr}{\mathscr{C}}

\newcommand{\Nscr}{\mathscr{N}}

\newcommand{\Fscr}{\mathscr{F}}
\newcommand{\Hscr}{\mathscr{H}}
\newcommand{\Lscr}{\mathscr{L}}

\newcommand{\Diff}{{\rm Diff}}

\newcommand{\dd}{\mathrm{d}}

\newcommand{\Mat}{{\rm M}}

\newcommand{\SL}{{\rm SL}}

\newcommand{\GL}{{\rm GL}}

\newcommand{\Mod}[1]{\left\vert{#1}\right\vert}
\newcommand{\nrm}[1]{\left\|#1\right\|}
\newcommand{\abs}[1]{\vert{#1}\vert}

\newcommand{\blob}{\rule[.2ex]{.8ex}{.8ex}\;}

\newcommand{\spec}{{\rm spec}}
\newcommand{\Irred}{\mathfrak{C}}

\newtheorem{defi}{Definition}

\newtheorem{lema}{Lemma}
\newtheorem{rmk}{Remark}

\newtheorem*{alg*}{Algorithm}
\newtheorem{prop}{Proposition}
\newtheorem{teor}{Theorem}
\newtheorem*{teor*}{Theorem}

\newcommand{\one}{{\bf 1}}
\newcommand{\FF}{\mathcal{F}}
\newcommand{\coment}[1]{}

\begin{document}

\title[Approximating Lyapunov Exponents]{Approximating Lyapunov Exponents and Stationary Measures}

\author[A.~Baraviera]{Alexandre Baraviera}
\address{ Instituto de Matem\'atica e Estat\'{\i}stica-UFRGS}
\email{baravi@mat.ufrgs.br}

\author[P.~Duarte]{Pedro Duarte}
\address{Centro de Matem\'atica, Aplica\c{c}\~{o}es Fundamentais e Investiga\c{c}\~{a}o Operacional, Faculdade de Ci\^{e}ncias, Universidade de Lisboa, 1749-016 Lisboa, Portugal}
\email{pmduarte@fc.ul.pt}

\date{\today}

\maketitle


\begin{abstract}
We give a new proof of E. Le Page's theorem on the H\"older continuity
of the first Lyapunov exponent in the class of irreducible Bernoulli cocycles.
This suggests an algorithm to approximate the first Lyapunov exponent,
as well as the stationary measure, for such random cocycles.

Keywords: Lyapunov exponent, random cocycle, stationary measure.

AMS subject classification: 37H15, 37D25.
\end{abstract}

\section{Introduction }

The description of the behavior of random linear cocycles is a classical subject studied in different
mathematical fields: it can be seen as a non-commutative random walk in Probability
 Theory, it relates to discrete Schr\"odinger operators describing a particle under a random
 potential in Quantum Mechanics, and it can also be regarded as toy model for the differential's dynamics of an ergodic diffeomorphism over
 a compact manifold in Dynamical Systems.
  An important feature in all these settings are the Lyapunov exponents (LE),
describing the exponential growth of norms of vectors under the action of the linear cocycle.

In Probability Theory the top Lyapunov exponent of a random Bernoulli cocycle measures the asymptotic behaviour $L_1:=\lim_{n\to +\infty} \frac{1}{n}\log \nrm{M_n} $ of the matrix products
$M_n=X_{n-1}\,\cdots \, X_1\, X_0 $ of   i.i.d. processes $\{X_n\}_{\geq 0}$ with values in some matrix group like $\GL_d(\R)$.
In this context, H. Furstenberg~\cite{F} gave an explicit integral formula for the largest Lyapunov exponent in terms of a stationary measure
for the action of the i.i.d. process $\{X_n\}$ on the projective space $\Pp(\R^d)$.
He also gave simple sufficient conditions
 for the top Lyapunov exponent to be non zero.

Such random linear cocycles can be described by the choice of a compact metric space $\Sigma$, a Borel probability measure $\mu$ on $\Sigma$ and a
measurable function $A\colon \Sigma\to\GL_d(\R)$.
If $\{Z_n\}_{n\geq 0}$ is a $\Sigma$-valued i.i.d. process with common distribution $\mu$, then
$X_n=A(Z_n)$ is an i.i.d. $\GL_d(\R)$-valued process
which determines a random Bernoulli cocycle.

 A natural question that arises is the continuity of the dependence of the top Lyapunov exponent $L_1$ as
 a function of  $A$ and  $\mu$.  A related important question is how to get good estimates for  the top Lyapunov exponent of a given cocycle. Because  Furstenberg's formula depends on a stationary measure which typically is not known in any explicit way,
 this problem has no obvious solution. The issue here is precisely to  estimate the stationary measure.  
Similar results were obtained recently by S. Galatolo et. al. (see \cite{GMN,GN}),  
regarding the problem of approximating  invariant measures, but working directly with transfer operators acting on measures and densities.

  Fixing the measure $\mu$ and assuming that the matrices $A$ preserve some cone family  (which means that the cocycle is uniformly hyperbolic)  Ruelle~\cite{Rue} was able to show that the top Lyapunov exponent
 depends analytically on $A$. In this setting M. Pollicott~\cite{Pol} obtained also
 a very efficient method to approximate the exponent numerically.

 On the other hand, dropping the uniform hyperbolicity assumption makes the continuity issue much more subtle and less regular. E. Le Page~\cite{LePage} was
able to show, under some general irreducibility assumption, a H\"older continuous dependence of the top Lyapunov exponent
as a function of  $A$. In this same setting, an example due to B. Halperin  (see Simon-Taylor~\cite{ST}) shows that the H\"older modulus of continuity can not be improved.

 Now if $\Sigma=\{1,\ldots, k\}$ is finite then the function $A$ takes a finite number of values $A_1,\ldots, A_k$. Considering a probability measure
  $\mu = p_1  \delta_{1} + \cdots + p_k \delta_{k}$ on $\Sigma$ with  $p_i > 0$ for all $i=1, \ldots, k$
  and $\sum_{j=1}^k p_j=1$,
 Y. Peres~\cite{Per} was able to prove the analiticity of the top Lypaunov exponent continuity with respect to the measure $\mu$.

 In this text we revisit these continuity results
  dealing with the dependence on $A$ and on $\mu$ in a unified way,  re-obtaining Le Page's result with a simpler proof and partially improving on Peres' result (see Theorem~\ref{teor joint cont}  and Remark~\ref{rmk LePage, Peres}).
We work with the adjoint of the usual transfer operator acting on probability measures. Under a suitable irreducibility assumption, this adjoint operator, still referred as a transfer operator, acts on spaces of H\"older continuous observables with nice contracting properties: it is a quasi-compact operator with simple maximal eigenvalue~\cite{Bou88,LePage}.
 The technique
 gives a method to approximate the stationary measure in Furstenberg formula when the original transfer operator
 is replaced by a finite-dimensional approximation (see Theorem~\ref{teor approx}), what also provides a way to estimate the top Lyapunov  exponent.   
At the end  we illustrate the method
with a couple of examples.

\bigskip

The paper is organized as follows:
In Section 2 we introduce the main concepts, definitions,
and we state our result on the continuity of the top Lyapunov exponent (Theorem~\ref{teor joint cont}). In Section 3 we define the main tool to deal with stationary measures, the so called transfer operators. Here we prove an abstract continuity theorem for transfer operators (see Theorem~\ref{teor abstr cont}). In sections 4 and 5 we prove Theorem~\ref{teor joint cont}.
In Section 6 we state and prove an approximation theorem (Theorem~\ref{teor approx}). We also describe a method to estimate the stationary measure and
the top Lyapunov exponent of a random cocycle.
In Section 7 we illustrate the approximation method with a couple of examples.
Section 8 is an appendix where we establish some geometric inequalities.

\bigskip

\section{Random linear cocycles}

A measure preserving transformation is a tuple $(T,\Omega,\Fscr,\Pp)$ where $T\colon \Omega\to \Omega$  is a bi-measurable automorphism of the measurable space $(\Omega,\Fscr)$, and
 $(\Omega,\Fscr,\Pp)$ is a probability space such that
$\Pp(T^{-1}E)=\Pp(E)$ for all $E\in\Fscr$. Such a transformation $(T,\Omega,\Fscr,\Pp)$
is said to be {\em ergodic} when $\Pp(E)=0$ or $\Pp(E)=1$ for every $E\in\Fscr$ such that $T^{-1} E=E$.

We call {\em  linear cocycle} over $(T,\Omega,\Fscr,\Pp)$
to a map
$F_A\colon \Omega\times \R^d\to\Omega\times\R^d$
of the form $F_A(\omega, v):=(T\omega, A(x)\,v)$,
determined by a measurable function $A\colon \Omega \to \GL_d(\R)$. Since $F_A$ is determined by $A$, we will refer to $A$ as the linear cocycle.
The cocycle $A$ is called {\em integrable} when
$$ \E( \log_+ \nrm{A^{\pm 1}} ):=\int_\Omega
\log_+ \nrm{A^{\pm 1}}\, \dd \Pp  <+\infty .$$

The iterates $F_A^n$ of the cocycle $A$ are given by
$ F_A^n(\omega, v)= (T^n, A^n(\omega) v)$ where
$$
A^n(\omega) := \left\{\begin{array}{lll}
A(T^{n-1}\omega) \cdots A(T\omega)\, A(\omega) & \text{ if } & n\geq 0 \\
A(T^{-n}\omega)^{-1} \cdots A(T^{-2}\omega)^{-1}\, A(T^{-1} \omega)^{-1} & \text{ if } & n < 0
\end{array} \right. $$

The top Lyapunov exponent of a linear cocycle $A$ is the first of the following two limits established by H. Furstenberg and H. Kesten~\cite{FKe}:

\begin{teor*}[Furstenberg-Kesten]
Let $(T,\Omega,\Fscr,\Pp)$ be an ergodic transformation,
and $A\colon \Omega\to \GL_d(\R)$ an integrable measurable random variable.
Then the the following  limits exist   $\Pp$-almost surely,
\begin{align*}
\gamma_+(A) &= \lim_{n\to \pm \infty} \frac{1}{n}\,\log \nrm{A^n(\omega)}, \\
\gamma_-(A) &= \lim_{n\to \pm \infty} \frac{1}{n}\,\log \nrm{A^n(\omega)^{-1}}^{-1} .
\end{align*}
\end{teor*}

\bigskip

 Given a matrix $M\in\Mat_d(\R)$ its {\em  singular values} are the square roots of the eigenvalues of the positive semi-definite symmetric matrix $M^T M$. The sorted singular values of $M$ are denoted by
$$ s_1(M)\geq s_2(M)\geq \cdots \geq s_d(M) \geq 0 .$$
The first singular value is the matrix's norm, $s_1(M)=\nrm{M}$,    which reflects the maximum expansion factor of $M$'s action on the Euclidean space $\R^d$. Likewise, the last singular value $s_d(M)$ is the minimum expansion factor of $M$.
One has $s_d(M)=0$ if $M$ is non-invertible,
and $s_d(M)=\nrm{M^{-1}}^{-1}$ otherwise.

Given $k\in\N$, we denote by $\wedge_k M$
the $k$-th exterior power of $M$. This is a  matrix
which represents the action of $M$ on the space $\wedge_k\R^d$ of $k$-th exterior vectors of $\R^d$ (see~\cite{Sternberg}), and whose norm can be expressed in terms of singular values
$$ \nrm{\wedge_k M} = s_1(M)\, s_2(M)\,\cdots\, s_k(M) .$$
It follows that for all $i=1,\ldots, d$,
$$ s_i(M)= \frac{\nrm{\wedge_i M}}{\nrm{\wedge_{i-1} M}} . $$

Let now $A\colon \Omega\to \GL_d(\R)$ be an integrable cocycle over some ergodic transformation
$(T,\Omega,\FF,\Pp)$. The ordered {\em Lyapunov exponents} of $A$ are defined to be the $\Pp$-almost sure limits
$$ L_i(A) = \lim_{n\to +\infty} \frac{1}{n}\,\log s_i( A^n ) = \lim_{n\to +\infty} \frac{1}{n}\,\left( \log \nrm{\wedge_i A^n}  -
\log \nrm{\wedge_i A^n}   \right) , $$
where the right-hand-side limit  exists by Furstenberg-Kesten's theorem applied to the integrable  exterior power cocycles $\wedge_i A$.
One has of course
$$ \gamma_+(A)= L_1(A)\; \text{ and }\;
\gamma_-(A)= L_d(A) .$$

From now on we will use only the notation $L_1(A)$
for the top Lyapunov exponent.

\bigskip

Given  a compact metric space $(\Sigma,d)$ (the {\em symbol space})
consider the space of sequences $\Omega_\Sigma=\Sigma^\Z$ endowed with the product topology. The homeomorphism $T\colon \Omega_\Sigma\to\Omega_\Sigma$,
$T\{\omega_i\}_{i\in \Z} := \{\omega_{i+1}\}_{i\in \Z} $, is called the {\em full shift map}.

Denote by $\Prob(\Sigma)$ the space of Borel probability measures on $\Sigma$.
For a given measure $\mu \in \Prob(\Sigma)$ consider the  product probability measure
$\Pp_\mu = \mu^\Z$ on $\Omega_\Sigma$. Then
$(T,\Omega_\Sigma,\Bscr,\Pp_\mu)$ is an ergodic transformation, referred as a  {\em full Bernoulli shift}.

\bigskip

A probability $\mu\in\Prob(\Sigma)$ and a continuous
function $A\colon \Sigma\to\GL_d(\R)$   determine
a measurable function $\tilde A\colon \Omega_\Sigma\to \GL_d(\R)$, $\tilde A\{\omega_n\}_{n\in\Z} = A(\omega_0)$, and hence
a linear cocycle $F_{(A,\mu)}\colon \Omega_\Sigma\times\R^d\to\Omega_\Sigma\times\R^d$
over the Bernoulli shift $(T,\Omega_\Sigma,\Bscr_\Sigma,\Pp_\mu)$.
We refer to the cocycle $F_{(A,\mu)}$ as a {\em random cocycle}.  The pair
$(A,\mu)\in\Cscr(\Sigma,\GL_d(\R))\times \Prob(\Sigma)$
is also referred as a random cocycle. The $n$-th iterate   $F_{(A,\mu)}^n = F_{(A^n,\mu^n)}$ is the random cocycle determined by the pair  $(A^n,\mu^n)$ where $\mu^n := \mu\times \cdots \times \mu
 \in\Prob(\Sigma^n)$ and $A^n\colon \Sigma^n\to \GL_d(\R)$
 is the function
 $$ A^n(x_0,x_1,\ldots, x_{n-1})
 :=A(x_{n-1}) \cdots A(x_1)\, A(x_0).$$
The top Lyapunov exponent of the random cocycle $(A,\mu)$ will be denoted by  $L_1(A,\mu)$.

\bigskip

Given a matrix $A\in \GL_d(\R)$ we denote by $\Phi_A\colon \Pp(\R^d)\to \Pp(\R^d)$ its projective action.

\begin{defi}
A   measure $\nu \in \Prob( \Pp(\R^d))$ is called stationary w.r.t.   $(A,\mu)$ when
$$ \nu = \int_{\Sigma}  (\Phi_{A(x)})_\ast \nu \, d\mu(x) . $$
\end{defi}

\bigskip

\begin{prop}[Furstenberg's formula~\cite{Fur}]
\label{prop Furstenber's formula}
For any random cocycle $(A,\mu)$ there exist  stationary measures
$\nu\in\Prob(\Pp(\R^d))$ such that
\begin{equation}
\label{Furstenberg formula}
L_1(A,\mu)= \int_{\Sigma} \int_{\Pp(\R^d)} \log \nrm{ A(x)\, p}\, d\nu(\hat p)\, d\mu(x) .
\end{equation}
\end{prop}

\bigskip

\begin{defi}[see D\'{e}finition 2.7 in~\cite{Bou88}]
A cocycle $(A,\mu)$ is called quasi-irreducible  if
there is no proper subspace  $V\subset \R^d$ which is invariant under all matrices of the cocycle, i.e., such that
$A(x) V=V$ for $\mu$-a.e. $x\in \Sigma$, and where $L_1(A\vert_V)\leq L_2(A)$.
\end{defi}

\begin{rmk}
If a cocycle $(A,\mu)$ is quasi-irreducible then it admits a unique stationary measure $\nu\in \Prob(\Pp(\R^d))$. Thus, in this case $L_1(A,\mu)$ is uniquely determined by the probability $\nu$ through Furstenberg's formula~\eqref{Furstenberg formula}.
\end{rmk}

\bigskip

The space of cocycles $\Cscr(\Sigma,\GL_d(\R))$
is endowed with the distance
$$ d_\infty(A,B) = \max_{x\in \Sigma} \nrm{A(x)-B(x)} .$$

On the space $\Prob(\Sigma)$ we consider the total variation  metric, which is defined by
$$ d(\mu_1,\mu_2):= \nrm{\mu_1-\mu_2}   $$
where $\nrm{\mu}$ stands for the total variation of a measure $\mu$.

\bigskip

\begin{defi}
We define the space $\Irred_d(\Sigma)$   of all cocycles $(A,\mu)\in \Cscr(\Sigma,\GL_d(\R))\times\Prob(\Sigma)$ such that
\begin{enumerate}
\item $(A,\mu)$ is quasi-irreducible and
\item $L_1(A,\mu)>L_2(A,\mu)$.

\end{enumerate}

\end{defi}

\begin{teor}
\label{teor joint cont}
The function
 $L_1\colon \Irred_d(\Sigma) \to\R$ is

\begin{enumerate}
\item locally H\"older continuous  w.r.t. the metrics $d_\infty$ and $d$,
\item is locally Lipschitz continuous in the variable $\mu$ w.r.t. $d$.
\end{enumerate}
\end{teor}

\begin{proof}
Follows from Propositions~\ref{prop cont mat} and~\ref{prop cont prob}.
\end{proof}

\begin{rmk}
\label{rmk LePage, Peres}
The H\"older continuous dependence of $L_1(A,\mu)$ on $A$ is basically E. Le Page's~\cite[Th\'{e}or\`{e}me 1]{LePage}.
Item (2) improves Y. Peres's~\cite[Theorem 1]{Per} in the sense that measure here may have infinite support. Our result is of course weaker in the sense that we only prove Lipschitz continuity while Y. Peres proves analiticity of the LE.
\end{rmk}

A coarser metric (and topology)
can be introduced with respect to which
the top LE is still H\"older continuous.
Let $\Diff^0(\Sigma)$ denote the $\infty$-dimensional group of   homeomorphisms
$h\colon \Sigma\to\Sigma$. Given $\mu\in\Prob(\Sigma)$
its {\em $h$-pullback}  is the probability measure
$h^\ast \mu:=(h^{-1})_\ast \mu = \mu \circ h^{-1}$.
Analogously, given a  cocycle $(A,\mu)$ we define its {\em $h$-pullback} as
$(A\circ h, h^\ast \mu)$.
The cocycles $F_{(A,\mu)}$ and $F_{(A\circ h, h^\ast \mu)}$
are conjugated via the isomorphism
$H\colon \Omega_\Sigma\times\R^d\to \Omega_\Sigma\times\R^d$,
$H(\{\omega_n\}_{n\in\Z},v) := (\{h(\omega_n)\}_{n\in\Z},v)$.
In particular we have
$$ L_1(A,\mu) = L_1(A\circ h, h^\ast \mu) .$$
Thus if one defines the metric
$$ \rho( (A,\mu), (B,\nu) ):= \inf  \left\{ \, d_\infty(B,A\circ h) + d(\nu,h^\ast \mu)\,\colon \, h\in\Diff^0(\Sigma)\, \right\}$$
the function $L_1\colon \Irred_d(\Sigma)\to \R$ is still locally H\"older
w.r.t. $\rho$.

\section{Transfer Operators}

Let $X$ be a compact metric space.
Given $0\leq \alpha \leq 1$, we denote by
$\Hscr_\alpha(X)$ the space of $\alpha$-H\"older continuous functions on $X$. On this space consider the seminorm
$v_\alpha\colon \Hscr_\alpha(X)\to \R$ defined by
$$  v_\alpha(\varphi):= \mathrm{diam}(X)^\alpha\,\sup_{\substack{x\neq y\\ x,y\in X}} \frac{\abs{\varphi(x)-\varphi(y)}}{d(x,y)^\alpha} . $$
The space $\Hscr_\alpha(X)$ is a Banach space (in fact a a Banach algebra with unity~\cite[Proposition 5.4]{DK-book}) when endowed with the norm
$$ \nrm{\varphi}_\alpha:= v_\alpha(\varphi)+\nrm{\varphi}_\infty .$$

The family of semi-normed spaces $\{ (\Hscr_\alpha(X),v_\alpha) \}_{\alpha\in [0,1]}$ is a scale in  sense that for all
$0\leq \alpha  <\beta \leq 1$, $0\leq t\leq 1$ and $\varphi\in\Hscr_\alpha(X)$,
\begin{enumerate}
\item[(A1)] $\Hscr_{\alpha}(X)\subset \Hscr_{\beta}(X)$,
\item[(A2)] $v_{\alpha} (\varphi) \leq v_{\beta}(\varphi) $,
\item[(A3)] $ v_{(1-t) \alpha + t \beta}(\varphi) \leq
v_{\alpha }(\varphi)^{1-t}v_{\beta}(\varphi) ^t $.
\end{enumerate}

Remark that
$\Hscr_0(X)$ coincides with the space $\Cscr(X)$ of continuous functions on $X$.
Moreover, for any   $\varphi\in\Cscr(X)$ and $x_0\in X$,
$$ \nrm{\varphi-\varphi(x_0)}_\infty \leq v_0(\varphi)\leq \nrm{\varphi}_\infty .$$
We denote by $\one$ the constant function $1$.

\bigskip

\begin{defi}
A linear operator
$\Lscr\colon \Cscr(X)\to \Cscr(X)$ is called a Markov operator when for all $\varphi\in\Cscr(X)$,
\begin{enumerate}
\item  $\Lscr \one =\one$,
\item  $\varphi\geq 0 $ implies $\Lscr \varphi \geq 0$,
\item $\nrm{\Lscr \varphi}_\infty \leq \nrm{\varphi}_\infty$.
\end{enumerate}
\end{defi}

\begin{defi}
Given $0<\sigma <1$ and $0<\alpha\leq 1$, a Markov operator $\Lscr\colon \Cscr(X)\to \Cscr(X)$ is said to act $\sigma$-contractively on $\Hscr_\alpha(X)$ if
for all $\varphi\in\Hscr_\alpha(X)$
$$ v_\alpha(\Lscr \varphi) \leq \sigma\, v_\alpha(\varphi) .$$
\end{defi}

\begin{rmk} If  $\Lscr$
 acts $\sigma$-contractively on $\Hscr_\alpha(X)$
then $\Lscr\colon \Hscr_\alpha(X)\to \Hscr_\alpha(X)$ is a quasi-compact operator with simple maximal eigenvalue $\lambda=1$ (see~\cite{HH}).
\end{rmk}

\begin{defi}
Let $\Lscr\colon \Cscr(X)\to\Cscr(X)$ be a Markov operator.
We call $\Lscr$-stationary probability  any measure $\nu\in \Prob(X)$
such that for all $\varphi\in\Cscr(X)$,
$$ \int_X \Lscr(\varphi)\,d\nu = \int_X\varphi\,d\nu .$$

\end{defi}

\begin{teor} Let $\Lscr\colon \Cscr(X)\to \Cscr(X)$ be  a Markov operator. If for some $0<\alpha\leq 1$
and $0<\sigma <1$  the Markov operator $\Lscr$   acts $\sigma$-contractively on $\Hscr_\alpha(X)$ then there exists a (unique)  $\Lscr$-stationary measure $\nu\in\Prob(X)$ such that defining the subspace
$$\Nscr_\alpha(\nu):=\left\{\varphi\in  \Hscr_\alpha(X)\colon \int_X \varphi\, \dd  \nu  =0 \right\} $$
one has
\begin{enumerate}
\item $\spec(\Lscr\colon \Hscr_\alpha(X)\to \Hscr_\alpha(X))\subset \{1\}\cup  \Dsk_\sigma(0)$,
\item $\Hscr_\alpha(X)= \R\,\one \oplus \Nscr_\alpha(\nu)$
is a $\Lscr$-invariant decomposition,

\item  $\Lscr$ fixes every function in $\R \,\one$ and acts as a
contraction with spectral radius $\leq \sigma$ on $\Nscr_\alpha(\nu)$.
\end{enumerate}
\end{teor}

\begin{proof}
By assumption  $\Lscr$ acts on quotient
space $\Hscr_\alpha(X)/\R\one$ as a $\sigma$-contraction. Since $\Lscr$ also fixes the constant functions in $\R\one$, it is a quasi-compact operator with simple eigenvalue $1$ (associated to  eigen-space $\R\,\one$)
and inner  spectral radius $\leq \sigma$.
Hence  $\spec(\Lscr)\subset \{1\}\cup  \Dsk_\sigma(0)$.

By spectral theory~\cite[Chap. XI]{RN-FA} there exists a
$\Lscr$-invariant decomposition
$\Hscr_\alpha(X)=\R\,\one \oplus \Nscr_\alpha$
such that $\Lscr$   acts as a
contraction with spectral radius $\leq \sigma$ on $\Nscr_\alpha$. Thus we can define a linear functional
$\Lambda\colon \Hscr_\alpha(X)\to\R$ setting
$\Lambda(c\,\one + \psi):=c$ where $\psi\in\Nscr_\alpha$. This functional has several properties:\\
\blob  $\Lambda(\one)=1$;\\
\blob  $\Lambda$ is positive in the sense that $\varphi\geq 0$ implies $\Lambda(\varphi)\geq 0$, since
$\Lscr$ is a Markov operator.
Given $\varphi=c\,\one + \psi \geq 0$ with
$\psi\in \Nscr_\alpha$, we have
$0\leq \Lscr^n(c\,\one+\psi) = c + \Lscr^n(\psi) $
for all $n\geq 0$.
Since $\psi\in \Nscr_\alpha$  we have $\lim_{n\to+\infty} \Lscr^n(\psi)=0$  which implies $\Lambda(\varphi)=c\geq 0$; \\
\blob $\Lambda$ is continuous w.r.t. the norm $\nrm {\cdot}_\infty$.  Indeed given any function
$\varphi\in \Hscr_\alpha(X)$
using
$$ -\nrm{\varphi}_\infty \one \leq \varphi
\leq \nrm{\varphi}_\infty \one $$
by positivity of $\Lscr$ it follows that
$$ \abs{\Lambda(\varphi)} \leq \Lambda(\one)\,\nrm{\varphi}_\infty = \nrm{\varphi}_\infty . $$
\blob $\Lambda$ extends to positive linear functional  $\tilde{\Lambda}\colon \Cscr(X)\to\R$ because by Stone-Weierstrass theorem the algebra $\Hscr_\alpha(X)$ is dense in $\Cscr(X)$;\\
\blob Finally by Riesz Theorem there exists a Borel probability  $\nu\in\Prob(X)$ such that
$\tilde{\Lambda}(\varphi)=\int_X\varphi\,d\nu$
for all $\varphi\in\Cscr(X)$.

Since by definition   $\Nscr_\alpha$ is the kernel of $\Lambda$, the relation $\Nscr_\alpha=\Nscr_\alpha(\nu)$ holds.
\end{proof}

\bigskip

Let $(\Sigma,d)$ be another compact metric space
and fix a Borel probability measure $\mu\in\Prob(\Sigma)$.
Given a continuous map  $M\colon \Sigma\times X\to X$
we consider the Markov operator
$\Lscr_{M,\mu}=\Lscr_M\colon \Cscr(X)\to\Cscr(X)$,
$$(\Lscr_M \varphi)(p):=\int_X \varphi(M(x,p))\,\dd\mu(x) . $$

\begin{rmk}
\label{LM is Markov}
The linear map
$\Lscr_M\colon\Cscr(X)\to\Cscr(X)$ is a Markov operator.
\end{rmk}

Define the following quantity
\begin{equation}
\label{kapa alpha(M)}
\kappa_\alpha(M,\mu) := \sup_{\substack{p\neq q\\ p,q\in X}}\int_X \left(\frac{d(M(x,p),M(x,q))}{d(p,q)} \right)^\alpha\, \dd \mu(x)
\end{equation}
which measures the average H\"older constant of
the function $M$ in the second argument.

The importance of this measurement is highlighted by the following proposition
\begin{prop}
\label{prop valpha kappa valpha}
 For all $\varphi\in\Hscr_\alpha(X)$,
$$ v_\alpha(\Lscr_M(\varphi))\leq \kappa_\alpha(M,\mu)\,v_\alpha(\varphi) .$$
\end{prop}

\begin{proof} Given $\varphi\in\Hscr_\alpha(X)$,
and $p,q\in X$,
\begin{align*}
\Mod{ (\Lscr_M(\varphi)(p) -  (\Lscr_M(\varphi)(q)  }
&\leq
\int_X \Mod{\varphi(M(x,p)) - \varphi(M(x,q)) }\, d\mu(x)\\
&\leq v_\alpha(\varphi)\,
\int_X d(M(x,p),M(x,q))^\alpha\, d\mu(x)\\
&\leq v_\alpha(\varphi)\,\kappa_\alpha(M,\mu)\, d(p,q)^\alpha
\end{align*}
which proves the proposition.
\end{proof}

Next we define a distance between two functions
$M,M'\colon \Sigma\times X\to X$.

\begin{equation}
\label{Delta alpha(M,M')}
 \Delta_\alpha(M,M'):= \sup_{p\in X}\int_\Sigma d(M(x,p),M'(x,p))^\alpha\,\dd \mu(x) .
\end{equation}

\begin{teor}
\label{teor abstr cont}
Let $M, M'\colon \Sigma\times X\to X$ be continuous functions.
Assume that
  $\kappa:= \kappa_\alpha(M)  <1$ for some $0<\alpha\leq 1$.  Then for all $n\in\N$ and  $\varphi\in\Hscr_\alpha(X)$,
\begin{equation}
\label{LMn-LM'n}
 \nrm{ \Lscr_M^n (\varphi)- \Lscr_{M'}^n (\varphi) }_\infty  \leq \frac{\Delta_\alpha(M,M')}{1-\kappa}\, v_\alpha(\varphi) .
\end{equation}
Moreover, if also  $\kappa_\alpha(M')  <1$ then   for all $\varphi\in\Hscr_\alpha(X)$,
\begin{equation}
\label{int phi nu - int phi nu'}
\Mod{\int_X \varphi\,d\nu_M - \int_X \varphi\,d\nu_{M'} } \leq \frac{\Delta_\alpha(M,M')}{1-\kappa}\, v_\alpha(\varphi) .
\end{equation}
\end{teor}

\begin{proof}
First notice that
\begin{align}
 \nrm{\Lscr_M(\varphi)  - \Lscr_{M'}(\varphi)}_\infty &\leq \sup_{p\in X}\int_\Sigma \Mod{\varphi(M(x,p)-\varphi(M'(x,p))}\,\dd \mu(x)
 \nonumber\\
 &\leq v_\alpha(\varphi)\, \sup_{p\in X}\int_\Sigma d(M(x,p),M'(x,p))^\alpha\,\dd \mu(x)  \nonumber\\
 & = \Delta_\alpha(M,M')\, v_\alpha(\varphi)
 \label{LM-LM'} .
\end{align}

Then using~\eqref{LM-LM'} and the relation
$$ \Lscr_M^n  - \Lscr_{M'}^n
= \sum_{i=0}^{n-1} \Lscr_{M'}^i\circ (\Lscr_M-\Lscr_{M'})\circ \Lscr_M^{n-i-1} $$
we get
\begin{align*}
\nrm{ \Lscr_M^n (\varphi)- \Lscr_{M'}^n (\varphi) }_\infty &\leq \sum_{i=0}^{n-1} \nrm{ \Lscr_{M'}^i((\Lscr_M-\Lscr_{M'}) ( \Lscr_M^{n-i-1}(\varphi))) }_\infty\\
&\leq \sum_{i=0}^{n-1} \nrm{ (\Lscr_M-\Lscr_{M'}) ( \Lscr_M^{n-i-1}(\varphi)) }_\infty\\
&\leq \sum_{i=0}^{n-1} \Delta_\alpha(M,M')\, v_\alpha(\Lscr_M^{n-i-1}(\varphi)))  \\
&\leq \Delta_\alpha(M,M')\, v_\alpha( \varphi) \, \sum_{i=0}^{n-1} \kappa^{n-i-1}    \\
&\leq \frac{\Delta_\alpha(M,M')}{1-\kappa}\, v_\alpha(\varphi) .
\end{align*}
This proves~\eqref{LMn-LM'n}.  Finally, since
$\lim_{n\to +\infty} \Lscr_M^n(\varphi) =\left(\int_X\varphi\, \dd\nu_M\right)\one$   and
$\lim_{n\to +\infty} \Lscr_{M'}^n(\varphi) =\left(\int_X\varphi\, \dd\nu_{M'}\right)\one$, one has
\begin{align*}
\Mod{\int_X \varphi\,d\nu_M - \int_X \varphi\,d\nu_{M'} }  &\leq \sup_{n}
\nrm{ \Lscr_M^n(\varphi) - \Lscr_{M'}^n(\varphi)}_\infty \leq \frac{\Delta_\alpha(M,M')}{1-\kappa}\, v_\alpha(\varphi)
\end{align*}
which proves~\eqref{int phi nu - int phi nu'}.
\end{proof}

\bigskip

\section{Continuous dependence on matrices}

A random cocycle $(A,\mu)\in \Cscr(\Sigma,\GL_d(\R))\times\Prob(\Sigma)$
determines the continuous function
$M_A\colon \Sigma\times\Pp(\R^d)\to \Pp(\R^d)$ defined by
$$ M_A(x,\hat p):= \Phi_{A(x)}(\hat p)  $$
which we use to introduce the Markov operator
$$ \Lscr_{A, \mu}(\varphi)(\hat p):= \int_\Sigma \varphi(M_A(x,\hat p))\,\dd \mu(x).$$

The quantity ~\eqref{kapa alpha(M)} is in this case
$$ \kappa_\alpha(A,\mu):= \sup_{\hat p \neq \hat q}
\E_\mu \left[ \, \left(\frac{d(\Phi_A(\hat p), \Phi_A(\hat q) }{d(\hat p, \hat q) }\right)^\alpha\, \right] , $$
where we write $\E_\mu[ f]:=\int_\Sigma f\,\dd\mu $.
By Proposition~\ref{prop valpha kappa valpha}, this measurement is an upper-bound on the contractiveness  of the Markov operator $\Lscr_{A, \mu}$ on the H\"older space $\Hscr_\alpha(X)$.

\begin{prop}
\label{prop kappa alpha}
Let $(A,\mu)\in \Cscr(\Sigma,\GL_d(\R))\times\Prob(\Sigma)$ be a random cocycle. Then
$$ \kappa_\alpha(A,\mu) = \sup_{\hat p \in\Pp(\R^d)}
\E_\mu \left[ \, \nrm{(D \Phi_A)_{\hat p} }^\alpha\, \right] . $$
\end{prop}

\begin{proof}
See Proposition~\ref{prop alpha ineq} in the Appendix.
\end{proof}

\bigskip

A couple of lemmas are needed to prove Theorem~\ref{teor joint cont}.

\begin{lema}
\label{uniform L1}
Let $(A,\mu)$ be some quasi-irreducible cocycle such that $L_1(A,\mu)>L_2(A,\mu)$.
Then
$$ \lim_{n\to +\infty} \frac{1}{n}\,\E\left[ \,
\log \nrm{A^n \,p}
\,\right] = L_1(A,\mu)   $$
with uniform convergence in $\hat p \in \Pp(\R^d)$, and where
$p\in \hat p$ stands for a unit vector representative of $\hat p\in \Pp(\R^d)$.
\end{lema}

\begin{proof}
See~\cite[Lemme 3.1]{Bou88}
\end{proof}

The Markov operator $\Lscr_{A,\mu}$ does not depend continuously on either $A$ or $\mu$.
Nevertheless, next lemma shows that it acts in a  uniform and
contracting way (in fact locally uniform in both variables $A$ and $\mu$)   on the semi-normed space $(\Hscr_\alpha(X),v_\alpha)$, for some small enough $\alpha$
and some large enough iterate.

\begin{lema}
\label{uniform kappa alpha}
Let $(A_0,\mu_0)$ be a quasi-irreducible cocycle such that $L_1(A_0,\mu_0)>L_2(A_0,\mu_0)$.
There  are numbers $\delta>0$, $0<\alpha<1$, $0<\kappa <1$ and $n \in\N$ such that  for all
 $A\in\Cscr(\Sigma,\GL_d(\R))$ with $d_\infty(A,A_0)<\delta$,
and for all  $\mu\in\Prob(\Sigma)$ with  $d(\mu,\mu_0)<\delta$, one has \;
$\kappa_\alpha(A^n,\mu^{n}) \leq \kappa$.
\end{lema}

\begin{proof}
By formula~\eqref{D PhiA} in the Appendix (see also the proof of Proposition~\ref{incr ratio PhiA alpha})
$$ \nrm{(D\Phi_A)_{\hat p} }  \leq \frac{\nrm{\wedge_2 A}}{\nrm{A p}^2}. $$
Hence the derivative $\nrm{(D\Phi_{A(x)})_{\hat p}}$  is uniformly bounded in a neighbourhood of $A$ and, by Proposition~\ref{prop kappa alpha},
 the measurement $\kappa_\alpha(A,\mu)$ is continuous in both variables $A$ and $\mu$, w.r.t. to the metric $d_\infty$ in the space $\Cscr(\Sigma,\GL_d(\R))$ and the total variation distance in the space $\Prob(\Sigma)$.
For this reason we can, and will, assume that $A$ and $\mu$ are fixed.

We have
\begin{align*}
 \lim_{n\to +\infty} \frac{1}{n}	\,\E_\mu[\, \log \nrm{(D\Phi_{A^n})_{\hat p} }
 \, ]&\leq \lim_{n\to +\infty}
  \frac{1}{n}	\,\E_\mu[\, \log \nrm{\wedge_2 A^n}\,] - 2\,\frac{1}{n}	\,\E[\, \log \nrm{A^n p}\, ] \\
&= (L_1(A)+L_2(A))-2\,L_1(A) = L_2(A) -L_1(A) <0 .
\end{align*}
Since the convergence of the upper bound $\E[\,\log (\nrm{\wedge_2 A}/\nrm{A p}^2)\,]$ is uniform in $\hat p$, for some $n$ large enough we have for all $\hat p\in\Pp(\R^d)$
$$ \E_\mu\left[\, \log \nrm{(D\Phi_{A^n})_{\hat p} } \,\right]
  \leq   - 1 .$$

To finish the proof, using the following inequality
\begin{equation*}
\label{exp(x) ineq}
e^x\leq   1+ x +\frac{x^2}{2}\,e^{\abs{x}}
\end{equation*}
we get (uniformly in $\hat p$)
\begin{align*}
\E_\mu\left[ \,  \nrm{(D\Phi_{A^n})_{\hat p} }^\alpha \, \right] &\leq
\E_\mu\left[ \, e^{\alpha\,\log \nrm{(D\Phi_{A^n})_{\hat p} } } \, \right]\\
&\leq
\E_\mu\left[ \, 1+ \alpha\, \log \nrm{(D\Phi_{A^n})_{\hat p} }  +
\frac{\alpha^2}{2} \,
 \nrm{(D\Phi_{A^n})_{\hat p}  } \,
 \log^2 \nrm{(D\Phi_{A^n})_{\hat p} } \, \right]\\
&\leq 1-\alpha + K\,\frac{\alpha^2}{2}
\end{align*}
for some positive constant $K=K(A,n)$.
Thus, taking  $\alpha$ small enough we have
\begin{equation*}
\label{kappa alpha(A) bound}
\kappa_\alpha (A^n)\leq \kappa :=  1-\alpha + K\,\frac{\alpha^2}{2}<1 .
\end{equation*}
\end{proof}

\bigskip

The measurement~\eqref{Delta alpha(M,M')}
applied to cocycles leads to the following quantity
\begin{defi}
$$ \Delta_\alpha(A,B) := \sup_{\hat p\in \Pp(\R^d)}
\E_\mu\left[\,  d(\Phi_{A}(\hat p),
\Phi_{B}(\hat p) )^\alpha  \,\right] .$$
\end{defi}

\begin{rmk} Given random cocycles $(A,\mu)$ and $(B,\mu)$
over the same Bernoulli shift,
$$ \Delta_\alpha(A,B)\leq d_\infty(A,B)^\alpha . $$
\end{rmk}

\bigskip

\begin{prop}
\label{prop cont mat}
Let $(A_0,\mu_0)$ be a quasi-irreducible cocycle with $L_1(A_0,\mu_0)>L_2(A_0,\mu_0)$.
Then 	there are  positive constants
$\alpha$, $C$ and $\delta$ such that for all
$B_1,B_2\in \Cscr(\Sigma,\GL_d(\R))$ and $\mu\in\Prob(\Sigma)$
if $d_\infty(B_j,A_0)<\delta$, $j=1,2$, and   $d(\mu,\mu_0)<\delta$ then
$$ \Mod{L_1(B_1,\mu)-L_1(B_2,\mu)} \leq C\, d_\infty(B_1,B_2)^\alpha  .$$
\end{prop}

\begin{proof}
Given a matrix $A\in\GL_d(\R)$ let us write
$$ \varphi_A(\hat p ):= \log \nrm{A\, p} $$
where $p\in\hat p$ stands for a unit representative.

The function $\GL_d(\R)\ni A\mapsto \varphi_A\in \Hscr_1(\Pp(\R^d))$
is locally Lipschitz. Given $R>0$ there is a positive constant $C=C_R$ such that
$$ \nrm{\varphi_A-\varphi_B}_\infty \leq C_R\,\nrm{A-B} $$
for all matrices $A,B\in\GL_d(\R)$ such that $\max\{ \nrm{A},\nrm
{B}, \nrm{A^{-1}},\nrm{B^{-1}} \}\leq R$.

Consider now two nearby random quasi-irreducible cocycles $A$ and $B$,
over the same full Bernoulli shift, and assume both these cocycles have a gap between their first and second Lyapunov exponents.
We denote by $\nu_A$ and $\nu_B$ the respective (unique) stationary measures. By Lemma~\ref{uniform kappa alpha} there
 exist  $n\in\N$, $0<\alpha$ and $0<\kappa<1$ such that $\kappa_\alpha(A^n,\mu^n)\leq \kappa$ for all cocycles $(A,\mu)$ near $(A_0,\mu_0)$. Since the maps
$A\mapsto A^n$ and $\mu\mapsto \mu^n$ are locally Lipschitz
we can without loss of generality suppose that
$\max\{\kappa_\alpha(A,\mu),\kappa_\alpha(B,\mu)\}\leq \kappa$, i.e.,
take $n=1$.

Then, using Furstenberg's formula
\begin{align*}
\abs{L_1(A,\mu)-L_1(B,\mu)}&\leq
\E_\mu\left[\, \abs{ \smallint \varphi_A\,\dd\nu_A -
\smallint \varphi_B\,\dd\nu_B } \,\right]\\
&\leq
\E_\mu\left[\, \abs{ \smallint \varphi_A\,\dd\nu_A -
\smallint \varphi_A\,\dd\nu_B } \,\right] + \E_\mu\left[\, \abs{ \smallint \varphi_A\,\dd\nu_B -
\smallint \varphi_B\,\dd\nu_B } \,\right] \\
&\leq
\frac{\Delta_\alpha(A,B)}{1-\kappa}\,v_\alpha(\varphi_A) + \E_\mu\left[\,   \smallint \abs{ \varphi_A  -
 \varphi_B}\,\dd\nu_B   \,\right] \\
 &\leq
\frac{v_1(\varphi_A)}{1-\kappa}\, d_\infty(A,B)^\alpha+ C_R\, d_\infty(A,B) \\
\end{align*}
where $R$ is a uniform bound on the norms of the matrices $A(x)$, $B(x)$ and their inverses. This proves that $L_1$ is locally H\"older continuous in a neighbourhood of $A_0$.
\end{proof}

\bigskip

\section{Continuous dependence on probabilities}

Throughout the rest of this section let $(A_0,\mu_0)$ be a quasi-irreducible cocycle such that $L_1(A_0,\mu_0)>L_2(A_0,\mu_0)$.
Take positive constants $\delta>0$, $0<\alpha<1$, $0<\kappa <1$ and $n \in\N$ as given by Lemma~\ref{uniform kappa alpha}.

\begin{lema}
For all $A\in \Cscr(\Sigma,\GL_d(\R))$ with $d_\infty(A,A_0)<\delta$
$\mu_1, \mu_2\in\Prob(\Sigma)$ with $d(\mu_j,\mu_0)<\delta$  for $j=1,2$,  $\varphi\in\Hscr_\alpha(\Pp(\R^d))$,

$$
  \nrm{ \Lscr_{A, \mu_1} \varphi - \Lscr_{A, \mu_2} \varphi }_\infty  \leq
  d(\mu_1, \mu_2) \, v_{\alpha}(\varphi) .
$$
\end{lema}
\begin{proof}
We have that
\begin{align*}
&\nrm{ \Lscr_{A, \mu_1} \varphi  - \Lscr_{A, \mu_2} \varphi }_\infty
  \leq \sup_{\hat{p}}{ |  \Lscr_{A, \mu_1} \varphi(\hat{p}) - \Lscr_{A, \mu_2} \varphi(\hat{p})   | }   \\
  &\leq \sup_{\hat{p}}{ |  \int \varphi(\Phi_{A(g)}(\hat{p})) d\mu_1(g) -  \int \varphi(\Phi_{A(g)}(\hat{p}))   d\mu_2(g)  | }    \\
&=
  \sup_{\hat{p}}{ \Mod{  \int \varphi(\Phi_{A(g)}(\hat{p})) d(\mu_1 -\mu_2)(g)  \, } }  \\
& =  \sup_{\hat{p}}  \left|  \int    (\varphi(\Phi_{A(g)}(\hat{p})) - \varphi(\Phi_{A(g_0)}(\hat{p}) ))\, d(\mu_1-\mu_2)(g)  \;+ \right. \\
& \qquad\qquad  \left.     \varphi(\Phi_{A(g_0)}(\hat{p}) )   \, d(\mu_1-\mu_2)(g)        \right|    \\
&=  \sup_{\hat{p}}{ \abs{   \int
     ( \varphi(\Phi_{A(g)}(\hat{p}) )  -\varphi(\Phi_{A(g_0)}(\hat{p}))) \, d (\mu_1 -\mu_2)(g)      } }  \\
 &\leq    v_0(\varphi)\, d(\mu_1, \mu_2)    \leq  v_{\alpha}(\varphi) \, d(\mu_1, \mu_2).
\end{align*}

\end{proof}

\begin{lema}
For all $A\in \Cscr(\Sigma,\GL_d(\R))$ with $d_\infty(A,A_0)<\delta$,
$\mu_1, \mu_2\in\Prob(\Sigma)$ with $d(\mu_j,\mu_0)<\delta$  for $j=1,2$,   $\varphi\in\Hscr_\alpha(\Pp(\R^d))$ and  $n\in\N$,
$$
   \nrm{ \Lscr_{A, \mu_1}^n(\varphi) - \Lscr_{A, \mu_2}^n(\varphi)}_{\infty}  \leq
   \frac{  d(\mu_1,  \mu_2)  }{1-\kappa } v_{\alpha} (\varphi)
$$
Moreover, if also $\kappa_{\alpha}(A, \mu_2) < 1$ then for all $\varphi \in \Hscr_{\alpha}(\Pp(\R^d))$
$$
    \left|  \int_{\Pp(\R^d)} \varphi d\nu_1 - \int_{\Pp(\R^d)} \varphi d \nu_2    \right| \leq
       \frac{ d(\mu_1,  \mu_2) }{1-\kappa} v_{\alpha}(\varphi)
$$
where $\nu_i\in \Prob(\Pp(\R^d))$ is the stationary measure
of $(A,\mu_i)$, for $i=1,2$.
\end{lema}
\begin{proof}
 We know that
$$
  \Lscr_{A, \mu_1}^n - \Lscr_{A, \mu_2}^n = \sum_{i=0}^{n-1} \Lscr_{A, \mu_2}^i
  (\Lscr_{A, \mu_1} - \Lscr_{A, \mu_2} ) \Lscr_{A, \mu_1}^{n-i-1}
$$
Hence
\begin{align*}
\nrm{ \Lscr_{A, \mu_1}^n(\varphi) - \Lscr_{A, \mu_2}^n(\varphi)}_{\infty}  &\leq
  \sum_{i=0}^{n-1} \nrm{ \Lscr_{A, \mu_2}^i((\Lscr_{A, \mu_1} - \Lscr_{A, \mu_2} )( \Lscr_{A, \mu_1}^{n-i-1}(\varphi)))  }_{\infty} \\
& \leq  \sum_{i=0}^{n-1} \nrm{ (\Lscr_{A, \mu_1} - \Lscr_{A, \mu_2} )( \Lscr_{A, \mu_1}^{n-i-1}(\varphi))  }_{\infty}  \\
&\leq
  \sum_{i=0}^{n-1}  d(\mu_1,  \mu_2)\,  v_{\alpha}( \Lscr_{A, \mu_1}^{n-i-1}(\varphi)) \\
  &\leq d(\mu_1, \mu_2) \, v_{\alpha}(\varphi)  \sum_{i=0}^{n-1} \sigma^{n-i-1} \leq
   \frac{  d(\mu_1,   \mu_2)   }{ 1-\kappa } \,v_{\alpha}(\varphi) .
\end{align*}

Proceeding  exactly as in the proof of~\eqref{int phi nu - int phi nu'} we get
$$
 \left|  \int  \varphi d\nu_1 - \int  \varphi d \nu_2    \right|  \leq
    \sup_n{ \nrm{ \Lscr_{A, \mu_1}^n(\varphi) - \Lscr_{A, \mu_2}^n(\varphi)}_{\infty}     }
     \leq \frac{ d(\mu_1,   \mu_2) }{1-\kappa} v_{\alpha}(\varphi)
$$
\end{proof}

\begin{prop}
\label{prop cont prob}
Given $(A_0,\mu_0)\in \Irred_d(\Sigma)$, there are  positive constants
$C$ and $\delta$ such that for all $A\in \Cscr(\Sigma,\GL_d(\R))$ and
$\mu_1,\mu_2\in \Prob(\Sigma)$,
 if   $d(\mu_j,\mu)<\delta$, $j=1,2$, and $d_\infty(A,A_0)<\delta$ then
$$ \Mod{L_1(A,\mu_1)-L_1(A,\mu_2)} \leq C\, d(\mu_1,\mu_2)   .$$
\end{prop}

\begin{proof}
As in the proof of Proposition~\ref{kapa alpha(M)} we will assume that $n=1$, for
the constant $n$ in Lemma~\ref{uniform kappa alpha}.
Using the Furstenberg's formula we get
\begin{align*}
& |L_1(A, \mu_1) - L_1(A, \mu_2) | =
 \Mod{ \E_{\mu_1} \left[  \int \varphi_A d\nu_1 \right] -  \E_{\mu_2} \left[  \int \varphi_A d\nu_2 \right]  } \\
& \leq  | \E_{\mu_1} \left[  \int \varphi_A d\nu_1 \right] - \E_{\mu_1} \left[  \int \varphi_A d\nu_2 \right] +
 \E_{\mu_1} \left[  \int \varphi_A d\nu_2 \right] - \E_{\mu_2} \left[  \int \varphi_A d\nu_2 \right] |  \\
 &\leq
 \E_{\mu_1} \left[ | \int \varphi_A d\nu_1 - \int \varphi_A d\nu_2   | \right] +
 \int | \int \varphi_A d\nu_2 | \, d |\mu_1 - \mu_2|  \\
 & \leq
    \log{\nrm{A}}\, d(\mu_1, \mu_2) \frac{1}{1-\kappa} v_{\alpha}(\varphi_A) + \log{\nrm{A}} \,d(\mu_1, \mu_2) .
\end{align*}
\end{proof}

\bigskip

\section{Approximating the stationary measure}

In this section we prove the approximation theorem (Theorem~\ref{teor approx}) mentioned in the introduction and describe a procedure to approximate the first Lyapunov exponent, as well as the stationary measure, for a random cocycle over a Bernoulli shift in finitely many symbols.

Throughout this section we assume that  $\Sigma=\{1,\ldots, k\}$ and $(\underline{A},\underline{p})$ is a random cocycle over the Bernoulli shift $T\colon \Omega_\Sigma\to\Omega_\Sigma$, where
$\underline{A}=(A_1,\ldots, A_k)$ is a list of matrices in $\GL_d(\R)$ and $\underline{p}=(p_1,\ldots, p_k)$ is a probability vector.

\subsection{An approximation theorem}
The {\em discretization} of a random cocycle
$(\underline{A},\underline{p})$ is  a pair $(\FF,\underline{f})$,
where $\FF\subset \Pp(\R^d)$ is a finite set and $\underline{f}=(f_1,\ldots, f_k)$ is a list of maps $f_j\colon \FF\to\FF$, $j=1,\ldots, k$.

 The discretization $(\FF,\underline{f})$ determines the Markov operator
$\Lscr_\FF\colon \R^\FF\to\R^\FF$,
$$ (\Lscr_\FF \varphi)(\hat v):=
\sum_{j=1}^k p_j\, \varphi(f_j(\hat v)) , $$
which can also be viewed as the stochastic $\FF\times \FF$ matrix
$P_\FF=(P_\FF(\hat w, \hat v)  )_{\hat w,\hat v\in\FF}$ with entries
$$P_\FF(\hat w, \hat v):= \sum_{f_j(\hat v)=\hat w} p_j.$$

Given $0<\alpha<1$, the {\em $\alpha$-error} of the discretization  is defined to
\begin{equation}
\label{def Delta alpha}
 \Delta_\alpha\left(\underline{A},\FF\right)
:= \max_{\hat v\in\FF}  \sum_{j=1}^k p_j \left(
d( \Phi_{A_j}(\hat v), f_j(\hat v) )
 \right)^\alpha .
\end{equation}

Define also $H_\alpha(\underline{A},\underline{p}) \colon \Pp(\R^d)\to \R$,
$$ H_\alpha( \underline{A}, \underline{p})(\hat v)=
\sum_{j=1}^k p_j\,\nrm{(D\Phi_{A_j})_{\hat v}}^\alpha   $$
and notice that by Proposition~\ref{prop kappa alpha},
\begin{equation}
\label{kappa alpha = max H alpha}
\kappa_\alpha(\underline{A},\underline{p})  =
\max_{\hat v\in\Pp(\R^d)} H_\alpha(\underline{A}, \underline{p})(\hat v)  .
\end{equation}

A stochastic matrix $P$ is called {\em mixing} when it has a single final class, which moreover is aperiodic (see~\cite[Theorem 1.31]{Walters}).
If a stochastic matrix $P$ is mixing then it has a unique
stationary probability vector, which is supported on the final class of $P$.

\begin{teor}
\label{teor approx}
Given $0<\alpha<1$, consider a random cocycle
 $(\underline{A},\underline{p})$ such that $\kappa = \kappa_\alpha(\underline{A},\underline{p})< 1$, and let
 $(\FF,\underline{f})$ be a discretization of $(\underline{A},\underline{p})$ with error $\Delta_\alpha = \Delta_\alpha\left(\underline{A},\FF\right)$.
 Assume also that the stochastic matrix $P_\FF$ is mixing and denote by $\nu_\FF$ the stationary probability vector of $P_\FF$.
 Then for all $\varphi\in\Hscr_\alpha(X)$,
\begin{equation}
\label{approx error}
\Mod{\int_{\Pp(\R^d)} \varphi\,\dd \nu_A - \int_\FF \varphi \, \dd \nu_\FF } \leq \frac{\Delta_\alpha}{1-\kappa }\, v_\alpha(\varphi) .
\end{equation}
\end{teor}

\begin{proof}
Consider the norm  $\nrm{\varphi}_\FF:=\max_{\hat v\in\FF} \abs{\varphi(\hat v)}$ in $\R^\FF$.
Notice that because $\Lscr_\FF$ is a Markov operator, for every $\varphi\in \R^\FF$\,
$\nrm{\Lscr_\FF(\varphi)}_\FF\leq \nrm{\varphi}_\FF $, and
\begin{align*}
 \nrm{\Lscr_{\underline{A}}(\varphi)  - \Lscr_{\FF}(\varphi)}_\FF &\leq \max_{\hat v\in \FF} \,\sum_{j=1}^k p_j\, \Mod{ \varphi(\Phi_{A_j}(\hat v))-\varphi(f_j(\hat v))} \\
 &\leq v_\alpha(\varphi)\, \max_{\hat v\in \FF}\,  \sum_{j=1}^k p_j\, d(\Phi_{A_j}(\hat v), f_j(\hat v))^\alpha   \nonumber\\
 & = \Delta_\alpha(\underline{A},\FF)\, v_\alpha(\varphi) .
\end{align*}
Using this and the formula
$$ \Lscr_{\underline{A}}^n  - \Lscr_{\FF}^n
= \sum_{i=0}^{n-1} \Lscr_{\FF}^i\circ (\Lscr_{\underline{A}}-\Lscr_{\FF})\circ \Lscr_{\underline{A}}^{n-i-1} $$
we get
\begin{align*}
\nrm{ \Lscr_{\underline{A}}^n (\varphi)- \Lscr_{\FF}^n (\varphi) }_\FF &\leq \sum_{i=0}^{n-1} \nrm{ \Lscr_{\FF}^i((\Lscr_{\underline{A}}-\Lscr_{\FF}) ( \Lscr_{\underline{A}}^{n-i-1}(\varphi))) }_\FF\\
&\leq \sum_{i=0}^{n-1} \nrm{ (\Lscr_{\underline{A}}-\Lscr_{\FF}) ( \Lscr_{\underline{A}}^{n-i-1}(\varphi)) }_\FF\\
&\leq \sum_{i=0}^{n-1} \Delta_\alpha(\underline{A},\FF)\, v_\alpha(\Lscr_{\underline{A}}^{n-i-1}(\varphi)))  \\
&\leq \Delta_\alpha(\underline{A},\FF)\, v_\alpha( \varphi) \, \sum_{i=0}^{n-1} \kappa^{n-i-1}    \\
&\leq \frac{\Delta_\alpha(\underline{A},\FF)}{1-\kappa}\, v_\alpha(\varphi) .
\end{align*}
Finally, since
$\lim_{n\to +\infty} \Lscr_{\underline{A}}^n(\varphi) =\left(\smallint \varphi\, \dd\nu_{\underline{A}}\right)\one$   and
$\lim_{n\to +\infty} \Lscr_{\FF}^n(\varphi) =\left(\smallint \varphi\, \dd\nu_{\FF}\right)\one$,
\begin{align*}
\Mod{\int  \varphi\,d\nu_{\underline{A}} - \int  \varphi\,d\nu_{\FF} }  &\leq \sup_{n}
\nrm{ \Lscr_{\underline{A}}^n(\varphi) - \Lscr_{\FF}^n(\varphi)}_\infty \leq \frac{\Delta_\alpha(\underline{A},\FF)}{1-\kappa}\, v_\alpha(\varphi)  .
\end{align*}
\end{proof}

\begin{rmk}
The previous theorem entails a procedure to compute weak approximations of the stationary measure $\nu_A$.
\end{rmk}

\bigskip

\subsection{Special bounds for $\SL_2$ cocycles  }
Let
$\underline{A}=(A_1,\ldots, A_k)\in\SL_2(\R)^k$.
In this setting $d=2$ and we denote by $\Pp$ the $1$-dimensional projective space $\Pp(\R^2)$.

\begin{prop} Given $\alpha\in (0,1)$
and a unit vector $x\in\R^2$,
\begin{equation}
\label{H alpha}
H_\alpha(\underline{A},\underline{p})(\hat x) =\sum_{j=1}^k p_j\frac{1}{\nrm{A_j x}^{2\alpha}} .
\end{equation}
\end{prop}

\begin{proof}
Given a matrix $M\in\SL_2(\R)$, and a unit vector $x\in\R^2$,
check that
$$ \nrm{ (D \Phi_M)_{\hat x}}=\frac{1}{\nrm{M\,x}^2} .$$
See formula (1) of section 5.14 in~\cite{Herman}.
\end{proof}

\begin{prop}
Given $\phi\in\Hscr_\alpha(\Pp)$, $\hat x, \hat y\in\Pp$,
$$ \frac{\abs{\Lscr_{\underline{A}}(\phi)(\hat x)-\Lscr_{\underline{A}}(\phi)(\hat y)}}{d(\hat x, \hat y)^\alpha}
\leq \frac{H_\alpha(\hat x) + H_\alpha(\hat  y)}{2}\,v_\alpha(\phi) \leq \kappa\,v_\alpha(\phi) ,$$
where $H_\alpha=H_\alpha(\underline{A},\underline{p})$ and $\kappa=\kappa_\alpha(\underline{A},\underline{p})$.
In particular, $v_\alpha(\Lscr_{\underline{A}}(\phi)\leq \kappa\,v_\alpha(\phi)$.
\end{prop}

\begin{proof}
See Proposition~\ref{incr ratio PhiA alpha} in the Appendix.
\end{proof}

\begin{prop} Let $H_\alpha=H_\alpha(\underline{A},\underline{p})$.
If
$\nrm{\underline{A}}_\infty := \max_{1\leq j\leq k} \nrm{A_j} $
then
$$ \abs{H_\alpha'(x)} \leq 2\,\alpha\,(\nrm{\underline{A}}_\infty)^{2(1+\alpha)} .$$
\end{prop}

\begin{proof}
Given $M\in\SL_2$, consider the function
$g_M\colon\Pp\to\R$, $g_M(x)=\nrm{M x}^{-2\alpha}$.
A simple calculation gives
$$ (D g_M)_x(v)= -2\,\alpha\,\frac{\langle M x,  M v  \rangle}{\nrm{M x}^{2(\alpha+1)}} .$$
 Thus $\Mod{g_M'(x)}\leq 2\,\alpha\, \nrm{M}^{2(\alpha+1)}$ and
 $$ \Mod{ H_\alpha'(x)} \leq \sum_{j=1}^k p_j \Mod{ g_{A_j}'(x)}
 \leq  2\,\alpha\, \sum_{j=1}^k p_j   \nrm{A_j}^{2(\alpha+1)} . $$
\end{proof}

\bigskip

\subsection{Approximating method for the LE}
Let $\nu\in\Prob(\Pp(\R^d))$ be the stationary measure of $\underline{A}$ and consider
the family of functions $\phi_j\colon \Pp\to\R$,
$$ \phi_j(x):= \log \nrm{A_j x} . $$
Define also
$\psi_j=\Lscr(\phi_j)$ for  $1\leq j\leq k$.

By Furstenberg's formula~\eqref{Furstenberg formula}
$$ L_1(\underline{A})=
\sum_{j=1}^k p_j\,\int_{\Pp(\R^d)}  \phi_j(x)\, \dd\nu(x) =
\sum_{j=1}^k p_j\,\int_{\Pp(\R^d)}  \psi_j (x)\, \dd\nu(x) .$$

Given any finite set $\FF\subset \Pp(\R^d)$,
which we will refer as a {\em mesh}, consider the discretization $(\FF,\underline{f})$ of $\Lscr$
where $\underline{f}=(f_1,\ldots, f_k)$ is the following list of functions $f_j\colon \FF\to\FF$.
For each $1\leq j\leq k$ and $\hat v\in \FF$,
$f_j(\hat v)$ is the point in $\FF$ that minimizes the distance to $\Phi_{A_j}(\hat v)$.
In this way the
discretization $(\FF,\underline{f})$ is determined by the mesh $\FF$.
Let $\nu_\FF$ be the corresponding stationary measure of the stochastic matrix $P_\FF$, i.e., of Markov operator $\Lscr_\FF$,
which can be viewed as a probability vector $\nu_\FF\in\R^\FF$.
Then the following number is an approximation of the exact value $\gamma^+(\underline{A})$.
\begin{equation}
\label{def gamma plus FF}
 L_1(\underline{A},\FF) :=
\sum_{j=1}^k p_j\,\sum_{x\in \FF} \psi_j (x)\, \nu_\FF(x) .
\end{equation}
By Theorem~\ref{teor approx}, the error in this approximation is bounded by
\begin{equation}
\label{approx formula}
 \Mod{L_1(\underline{A}) - L_1(\underline{A},\FF) }
\leq \frac{\Delta_\alpha(\FF,\underline{A})}{1-\kappa_\alpha}\, V_\alpha(\underline{A},\FF),
\end{equation}
where
$$ \kappa_\alpha := \kappa_\alpha(\underline{A},\underline{p}) ,$$
and
$$  V_\alpha(\underline{A},\FF):=  \sum_{j=1}^k p_j\, v_\alpha(\psi_j) . $$
The advantage in using the functions $\psi_j=\Lscr(\phi_j)$ instead of
$\phi_j$ is that the H\"older constant $v_\alpha(\psi_j)$ is in general significantly smaller than its upper bound $\kappa_\alpha\,v_\alpha(\phi_j)$, thus improving the final error estimate.

\bigskip

In the rest of this section we describe and comment
each of the steps to implement this approximating method.

\subsubsection{Choose $\alpha$ and some iterate of the cocycle $\underline{A}$}
One needs to find $\alpha \in (0,1)$
and an integer $n\in\N$ such that
$\kappa=\kappa_\alpha(\underline{A}^n,\underline{p}^n)<1$. By Lemma~\ref{uniform kappa alpha} this always possible.

For $\SL_2$-valued cocycles our strategy
was to plot the one variable function
$H_\alpha=H_\alpha(\underline{A},\underline{p})$ for several values of $\alpha$ until it became plausible that its maximum was $<1$. When this failed we increased the number of iterates and repeated the process.

Because the number of matrices in $\underline{A}^n$ grows exponentially with $n$ one can only iterate  the cocycle   a small number of times  before the whole scheme becomes computationally too expensive.
For $\alpha\approx 1$, since the function
$H_1$ has mean value $1$, one has $\kappa_\alpha>1$.
For $\alpha\approx 0$, one has $H_\alpha\approx H_0\equiv 1$
so that $\kappa_\alpha\approx 1$. Hence the optimal choice of $\alpha$, if one wants to minimize $\kappa_\alpha$, lies somewhere between $0$ and $1$.
When $\alpha\approx 0$ we have $\Delta_\alpha\approx 1$ and the bound~\eqref{approx formula} is not so good. Similarly if $\kappa_\alpha\approx 1$ the denominator in the bound~\eqref{approx formula} becomes too small.
These constraints pose severe limitations on the class of cocycles to which this method can efficiently applied.

\subsubsection{Estimate $\kappa_\alpha$}
By~\eqref{kappa alpha = max H alpha} $\kappa_\alpha$ is the maximum of $H_\alpha=H_\alpha(\underline{A},\underline{p})$.

For $\SL_2$-cocycles, the maxima of the summunds $g_{A_j}(x):= \frac{1}{\nrm{A_j x}^{2\alpha}}$ in~\eqref{H alpha} are attained at the projective points
corresponding to the least expanding  singular directions of the matrices $A_j$.
 Splitting this data into clusters of nearby points, the barycenters of these clusters give us   best places where to search for the  local maxima of the function  $H_\alpha$.
In our opinion, using  a gradient method to find the local maxima near these clusters,
or else  a Newton method to compute the zeros of $H_\alpha'$, are  efficient schemes to estimate the global maximum
$$ \kappa_\alpha= \max_{\hat x\in\Pp} H_\alpha(\hat x) . $$

Because it was not   our goal to do rigorous numerics,
we didn't implement this scheme. Instead we used the  general purpose
function {\rm NMaximize} of {\em Mathematica} to approximate the absolute maximum of the one variable function $H_\alpha$.

\subsubsection{Choose a mesh $\FF$}
For instance a uniformly distributed mesh in $\Pp(\R^d)$. The bound on the number of mesh points should be determined in order to have an efficient computation of the stationary measure $\nu_\FF$.

\subsubsection{Compute the discretization determined by  $\FF$} This step is straightforward to implement.
We wrote its {\em Mathematica} code using the builtin function  {\rm Nearest[$data$, $x$]} which returns the nearest element to a number $x$ in a given list of real numbers $data$.

\subsubsection{Compute the stationary measure    $\nu_\FF$} There are many ways to approximate the stationary measure of a given stochastic matrix,
for instance by iteration of the stochastic matrix. We have used instead the builtin  function {\rm StationaryDistribution} of  {\em Mathematica}.

\subsubsection{Compute the Lyapunov exponent approximation    $L_1(\underline{A},\FF)$}
This step is straightforward to implement. By~\eqref{def gamma plus FF} this  involves adding up $k\cdot\Mod{\FF}$ terms.

\subsubsection{Estimate the $\alpha$-error bound    $\Delta_\alpha(\FF,\underline{A})$}
This step is also straightforward to implement. By~\eqref{def Delta alpha} this  involves maximizing a function over $\FF$.

\subsubsection{Estimate the average H\"older constant  $V_\alpha(\FF,\underline{A})$}
This is the critical step in  computational time costs. One has to estimate the H\"older constant $v_\alpha(\psi)$ for the functions
$\psi=\psi_j\colon\Pp(\R^d)\to \R$, $j=1,\ldots, k$.

For $\SL_2$ cocycles we have $d=2$, and one has to address the problem of estimating the
H\"older constant $v_\alpha(\psi)$ of a smooth function $\psi\colon\Pp\to\R$.
Denote by $\Sigma=\Sigma(\psi)\subset \Pp$ the finite set of all maxima and minima of $\psi$.
The procedure described in the step 6.3.2
may also be used to numerically approximate the extreme point   sets
$\Sigma_j:=\Sigma(\psi_j)$.
Define
$$ v_\alpha(\psi;\Sigma):= \max_{\substack{x,y\in\Sigma\\x\neq y}} \frac{\abs{\psi(x)-\psi (y)}}{d(x,y)^\alpha}. $$
The measurement $v_\alpha(\psi;\Sigma)$ is\ computable.
A problem subsists because in general
 $$ v_\alpha(\psi;\Sigma) <  v_\alpha(\psi) . $$
To estimate $v_\alpha(\psi)$,  find the pairs $(x_j,y_j)\in \Sigma(\psi)$, $j=1,\ldots, s$,
 where $x_j>y_j$ and
 $$ \frac{\abs{\psi(x_j)-\psi (y_j)}}{d(x_j,y_j)^\alpha} = v_\alpha(\psi;\Sigma) . $$
Take  each of these pairs as input in the following iterative scheme: Consider the sort of Newton method defined by
$ N_\alpha \colon ( x_0, y_0) \mapsto (x_1,y_1) $ where
\begin{align*}
x_1 &:= x_0 + \frac{1}{\psi''(x_0)}\, \left(\alpha\,\frac{\psi(y_0)-\psi(x_0)}{y_0-x_0} -\psi'(x_0)\right)  \\
y_1 &:= y_0 + \frac{1}{\psi''(y_0)}\, \left(\alpha\,\frac{\psi(y_0)-\psi(x_0)}{y_0-x_0} -\psi'(y_0)\right)  .
\end{align*}
An easy calculation shows that the critical points of the function $$ K_\alpha(x,y)=K_{\alpha,\psi}(x,y):= \frac{\psi(x)-\psi(y)}{(x-y)^\alpha} \quad (x>y) $$
are the points $(x,y)$ with $x>y$ such that
$$ \psi'(x) = \alpha\,\frac{\psi(x)-\psi(y)}{x-y} =\psi'(y) .$$
All these points are fixed points of the map $N_\alpha$. Moreover, the derivative of $N_\alpha$ vanishes at these critical points.
Hence, if $(x_0,y_0)$ is near a critical point of $K_\alpha$ then its iterates  $N_\alpha^n(x_0,y_0)$  converge  quadratically to a critical point $(x_\ast,y_\ast)$ of $K_\alpha$. In this way we can sharply approximate the absolute maxima of $K_{\alpha,\psi}$.

Because it was not   our goal to do rigorous numerics,
we didn't implement this method. Instead we used the general purpose
function {\rm NMaximize} of {\em Mathematica} to approximate the absolute maximum of the two variable function $K_{\alpha,\psi}(x,y)$.
Because in our applications we had to estimate the H\"older constants
$v_\alpha(\psi_j)$ for  the $k$ different functions $\psi_j$,
 the usage of {\em Mathematica} tool, instead of the scheme suggested above, was probably less efficient.

\subsubsection{Estimate the error bound in~\eqref{approx formula}.}
Simply combine the outputs of the steps 6.3.2, 6.3.7 and 6.3.8.

\bigskip

\section{Examples}
In the examples below we consider the following three families of matrices in $\SL_2(\R)$.
$$S_\lambda:= \begin{bmatrix}
\lambda & -1 \\ 1 & 0
\end{bmatrix},\;  D_\lambda:=\begin{bmatrix}
\lambda & 0 \\ 0 & \lambda^{-1}
\end{bmatrix} \; \text{ and }\;
R_\lambda:=\begin{bmatrix}
\cos\lambda & -\sin \lambda \\
\sin \lambda & \cos \lambda
\end{bmatrix} .$$

\bigskip

\subsection*{First example}
Consider  the cocycle generated by the symmetric matrices
$$\left\{ R_{-\frac{j\,\pi}{m}}\,D_\lambda\,  R_{\frac{j\,\pi}{m}}
\colon 1\leq j\leq m \right\} ,$$
chosen with equal probability $1/m$, with $m=8$ and $\lambda=2.2$.
We iterate this cocycle $3$ times to get a cocycle
$\underline{A}=(A_1,\ldots, A_k)$ with $k=512$ (equi-probable) matrices.

\subsection*{Second example}
Consider the Bernoulli Schr\"odinger cocycle generated by the two matrices $\left\{ S_8, S_{1.9} \right\}$
chosen with equal probability $1/2$.
Notice that $S_8$ is  hyperbolic, while $S_{1.9}$ is elliptic.
Hence this cocycle is not uniformly hyperbolic.
We iterate this cocycle $9$ times to get a cocycle
$\underline{A}=(A_1,\ldots, A_k)$ with $k=512$ (equi-probable) matrices.

\subsection*{Third example}
Consider the Bernoulli  cocycle generated by the  matrices $\left\{ D_{3.5}, R_{0.4} \right\}$
chosen with equal probability $1/2$.
This cocycle is not uniformly hyperbolic.
We iterate this cocycle $9$ times to get a cocycle
$\underline{A}=(A_1,\ldots, A_k)$ with $k=512$ (equi-probable) matrices.

\bigskip

\begin{figure}[h]
\begin{center}
\includegraphics[scale=0.5]{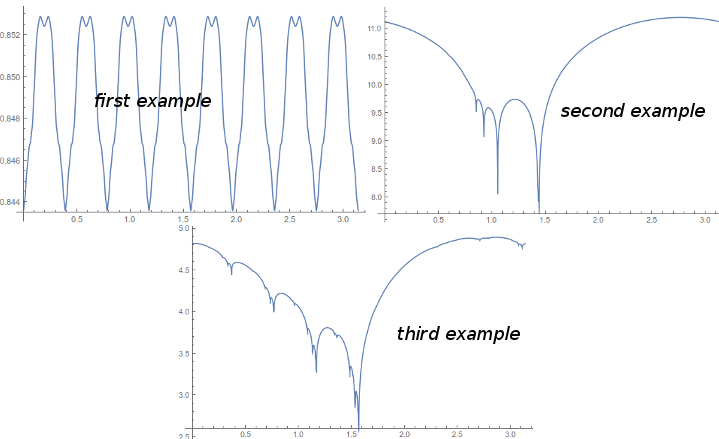}
\end{center}
\caption{Graphs of the functions  $\sum_{j=1}^k p_j\, \phi_j$}
\end{figure}

\begin{figure}[h]
\begin{center}
\includegraphics[scale=0.5]{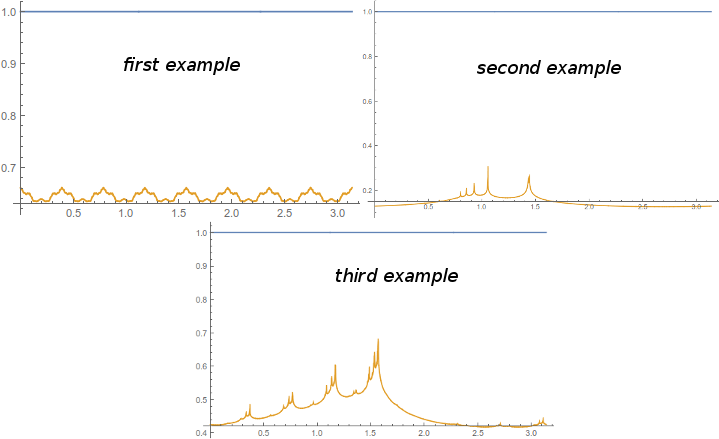}
\end{center}
\caption{Graphs of the functions  $H_\alpha(\underline{A})$}
\end{figure}

The numerics obtained are sinthesized in the following table.
We stress again that these computations do not involve any kind of rigorous error control.

\begin{table}[h]
\begin{tabular}{|c|c|c|c|}
\hline
 & $1^{{\rm st}}$ example & $2^{{\rm nd}}$ example & $3^{{\rm th}}$ example \\
\hline
\hline
$\alpha$ & $0.5$ & $0.1$ & $0.1$ \\
\hline
$\sigma_\alpha:= \max H_\alpha(\underline{A})$ &
$0.661033$ & $0.67282$ & $0.796829$ \\
\hline
$N=\Mod{\FF}$  & $1000$ & $512$ & $512$ \\
\hline
$\Delta_\alpha:=\Delta_\alpha(\underline{A},\FF)$ & $0.0285419$ & $0.517986656$ & $0.499914$ \\
\hline
$V_\alpha:=\sum_{j=1}^k p_j\, v_\alpha(\psi_j)$ & $0.284198$ & $0.0403991$ & $0.0642972$ \\
\hline
\hline
$L_1(\underline{A},\FF)$  & $0.8489$ & $10.7506$  & $4.56736$ \\
\hline
$ \Delta_\alpha\,(1-\sigma_\alpha)^{-1}\, V_\alpha$ & $0.0239301$ & $0.0639592$ &  $0.158207$ \\
\hline
\end{tabular}

\smallskip

\caption{Computed data for the three examples. The last two rows represent the approximate LE and the estimated upper bound on the error.}
\end{table}

\bigskip

\section{Appendix: some projective inequalities}

Consider the metric
$$ \delta(\hat x, \hat y):=\frac{\nrm{x\wedge y}}{\nrm{x} \nrm{y}}  $$
on the projective space $\Pp(\R^d)$, where $\hat x$ and $\hat y$ stand for projective classes of  non-zero vectors $x,y\in\R^d$.

Given a point $\hat x\in\Pp(\R^d)$
define the orthogonal projection $\pi_{\hat x} \colon\R^d\to\R^d$,
$$ \pi_{\hat x} (v):= v - \left( v\cdot x\right) \, x$$
onto the hyperplane $x^\perp$,  where $x\in\hat x$
is any unit vector representative of $\hat x$. Define also the (non linear) projection $\nu_{\hat x} \colon\R^d\to\R^d$,
$$ \nu_{\hat x}(v):= \frac{\pi_{\hat x}(v)}{\nrm{\pi_{\hat x}(v)}} .$$

\bigskip

Given a matrix  $A\in\GL_d(\R)$, let
$\Phi_A\colon \Pp(\R^d)\to \Pp(\R^d)$ be its projective action.

With the previous notation one has the following formula for the derivative of $\Phi_A$.
For any   $\hat x\in\Pp(\R^d)$ and $v\in x^\perp=T_{\hat x} \Pp(\R^d)$,
\begin{equation}
\label{D PhiA}
(D\Phi_A)_{\hat x}(v) = \frac{\pi_{\Phi_A(x)}(A\,v)}{\nrm{A x}} .
\end{equation}

\begin{rmk} \label{rmk diagonal}
Form the definition of derivative, given unit vectors $x,v\in \R^d$,
$$ \lim_{\hat y\to \hat x} \frac{\delta(\Phi_A(\hat x),\Phi_A(\hat y))}{\delta(\hat x,\hat y)}
= (D \Phi_A)_{\hat x}  (v)  $$
where the limit is taken over the projective line
$\mathrm{span}\{x,v\}\subset \Pp(\R^d)$.
\end{rmk}

\begin{prop}
\label{incr ratio PhiA alpha}
Given $\alpha>0$ and unit vectors $x,y\in \R^d$,
\begin{align*}
\left[ \frac{\delta(\Phi_A(\hat x),\Phi_A(\hat y))}{\delta(\hat x,\hat y)} \right]^\alpha
&\leq \frac{1}{2}\,\left\{
\nrm{ (D\Phi_A)_{\hat x} (\nu_{\hat x}(y))}^\alpha
+
\nrm{ (D\Phi_A)_{\hat y} (\nu_{\hat y}(x))}^\alpha \right\}\\
&\leq \frac{1}{2}\,\left\{
\nrm{ (D\Phi_A)_{\hat x} }^\alpha
+
\nrm{ (D\Phi_A)_{\hat y}}^\alpha \right\} .
\end{align*}
\end{prop}

\begin{proof}
Given  unit vectors $x,y\in \R^d$,
\begin{align*}
\left[ \frac{\delta(\Phi_A(\hat x),\Phi_A(\hat y))}{\delta(\hat x,\hat y)} \right]^\alpha
& =  \left[  \frac{\nrm{A x \wedge A y}}{\nrm{A x}\nrm{A y}}\,\frac{1}{\nrm{x \wedge y}} \right]^\alpha\\
& = \left[  \frac{\nrm{A x \wedge A y}}{\nrm{x \wedge y}}\right]^\alpha
\, \frac{1}{\nrm{A x}^\alpha}\,\frac{1}{\nrm{A y}^\alpha}  \\
&\leq \left[  \frac{\nrm{A x \wedge A y}}{\nrm{x \wedge y}}\right]^\alpha
\,\frac{1}{2}\,\left\{  \frac{1}{\nrm{A x}^{2\alpha}} + \frac{1}{\nrm{A y}^{2\alpha}} \right\} \\
&=
\frac{1}{2}\,\left\{  \left[  \frac{\nrm{A x \wedge A y}}{\nrm{x \wedge y}}\right]^\alpha\,\frac{1}{\nrm{A x}^{2\alpha}} +
\left[  \frac{\nrm{A x \wedge A y}}{\nrm{x \wedge y}}\right]^\alpha\,\frac{1}{\nrm{A y}^{2\alpha}} \right\} \\
\end{align*}
where we have used that
$ \sqrt{a\,b} \leq \frac{1}{2}\,\left\{a+b\right\} $
with $a=\nrm{A x}^{-2\alpha}$ and $b=\nrm{A y}^{-2\alpha}$.
On the other hand for any non-zero vector $u\in\R^d$,
because $\nrm{u\wedge w}$ is the area of the parallelogram
spanned by $u$ and $w$, we must  have
$\nrm{u\wedge w}= \nrm{u}\,\nrm{\pi_{\hat u}(w)}$. Hence
$$ \nrm{\pi_{\hat u}(w) }= \frac{\nrm{u\wedge w}}{\nrm{u}} . $$
Using this relation one has
\begin{align*}
\nrm{ (D\Phi_A)_{\hat x} (\nu_{\hat x}(y)) }  &=
\nrm{ (D\Phi_A)_{\hat x} \left(\frac{\pi_{\hat x}(y)}{\nrm{x\wedge y}} \right) } \\
&=
 \frac{ \nrm{ \pi_{\Phi_A(\hat x)} \left( A\,\pi_{\hat x}(y)  \right) } }{\nrm{A x} \,\nrm{x\wedge y}  }   =
 \frac{   \nrm{ A x \wedge  A\, \pi_{\hat x}(y)    } }{\nrm{A x}^2 \,\nrm{x\wedge y}  }  \\
& =
 \frac{   \nrm{ A x \wedge  A y    } }{\nrm{A x}^2 \,\nrm{x\wedge y}  }  .
\end{align*}
Similarly, exchanging the roles of $x$ and $y$,
$$\nrm{ (D\Phi_A)_{\hat y} (\nu_{\hat y}(x)) }  =
 \frac{   \nrm{ A x \wedge  A y    } }{\nrm{A y}^2 \,\nrm{x\wedge y}  }  .$$
 This establishes the proposition.
\end{proof}

\bigskip

\begin{prop}
\label{prop alpha ineq}
Let   $(\Omega,\Fscr,\Pp)$ be a probability space,
and $A\colon \Omega\to \GL_d(\R)$ a  matrix valued random variable. Then for any $\alpha>0$,
$$ \sup_{\hat x\neq \hat y} \,
\E\left[
\left( \frac{\delta(\Phi_A(\hat x),\Phi_A(\hat y))}{\delta(\hat x,\hat y)} \right)^\alpha   \right]
= \sup_{\hat x \in\Pp(\R^d)} \,
\E\left[ \,
\nrm{ (D\Phi_A)_{\hat x}}^\alpha  \, \right] .
$$

\end{prop}

\begin{proof}
For the first inequality ($\leq$) just average the
one in Proposition~\ref{prop alpha ineq}
and then take sup. The converse inequality ($\geq $) follows from Remark~\ref{rmk diagonal}.
\end{proof}

\bigskip

\section*{Acknowledgements}

The first author was partially supported by  CNPq through the project 312698/2013-5.

The second author was supported by the Research Centre of Mathematics of the University of Minho with the Portuguese Funds from the ``Funda\c c\~ao para a Ci\^encia e a Tecnologia", through the Project UID/MAT/00013/2013.


\providecommand{\bysame}{\leavevmode\hbox to3em{\hrulefill}\thinspace}
\providecommand{\MR}{\relax\ifhmode\unskip\space\fi MR }
\providecommand{\MRhref}[2]{%
  \href{http://www.ams.org/mathscinet-getitem?mr=#1}{#2}
}
\providecommand{\href}[2]{#2}

\end{document}